\documentclass[10pt]{amsart}
\usepackage{epsfig,enumerate,amsmath,amsfonts,amssymb}
\usepackage{amsthm}
\usepackage{epsf}
\usepackage{subfigure,psfrag}
\newtheorem{thm}{Theorem}[section]
\newtheorem{cor}[thm]{Corollary}

\newtheorem{defn}{Definition}[section]
\newtheorem{exam}[defn]{Example}
\newtheorem{rem}{Remark}[section]
\begin{document}

\title[]
{Coincidence and Common Fixed Point Results for Generalized $\alpha$-$\psi$ Contractive Type Mappings with Applications}
\author[Priya Shahi, Jatinderdeep Kaur, S. S. Bhatia]{Priya Shahi$^{\ast}$, Jatinderdeep Kaur, S. S. Bhatia}
 \thanks{$^{\ast}$Corresponding author: Priya Shahi, School of Mathematics and Computer Applications, Thapar University, Patiala 147004, Punjab, India. Email address:
 {\rm priya.thaparian@gmail.com}}
\maketitle
\begin{center}
{\footnotesize School of Mathematics and Computer Applications,\\Thapar University, Patiala-147004, Punjab, India.\\
Email addresses: priya.thaparian@gmail.com$^{\ast}$, jkaur@thapar.edu, ssbhatia@thapar.edu
\\}
\end{center}

\begin{abstract}
\noindent A new, simple and unified approach in the theory of contractive mappings was recently given by Samet \emph{et al.} (Nonlinear Anal. 75, 2012, 2154-2165) by using the concepts of $\alpha$-$\psi$-contractive type mappings and $\alpha$-admissible mappings in metric spaces. The purpose of this paper is to present a new class of contractive pair of mappings called generalized $\alpha$-$\psi$ contractive pair of mappings and study various fixed point theorems for such mappings in complete metric spaces. For this, we introduce a new notion of $\alpha$-admissible w.r.t $g$ mapping  which in turn generalizes the concept of $g$-monotone mapping recently introduced by $\acute{C}$iri$\acute{c}$ et al. (Fixed Point Theory Appl. 2008(2008), Article ID 131294, 11 pages). As an application of our main results, we further establish common fixed point theorems for metric spaces endowed with a partial order as well as in respect of cyclic contractive mappings. The presented theorems extend and subsumes various known comparable results from the current literature. Some illustrative examples are provided to demonstrate the main results and to show the genuineness of our results.
\end{abstract}


\section{Introduction}
Fixed point theory has fascinated many researchers since 1922 with the celebrated Banach fixed point theorem. There exists a vast literature on the topic and this is a very active field of research at present. Fixed point theorems are very important tools for proving the existence and uniqueness of the solutions to various mathematical models (integral and partial differential equations, variational inequalities etc). It is well known that the contractive-type conditions are very indispensable in the study of fixed point theory. The first important result on fixed points for contractive-type mappings was the well-known Banach-Caccioppoli theorem which was published in 1922 in \cite{Banach} and it also appears in \cite{Caccioppoli}. Later in 1968, Kannan \cite{Kannan} studied a new type of contractive mapping. Since then, there have been many results related to mappings satisfying various types of contractive inequality, we refer to (\cite{BhaskarandLakshmikantham}, \cite{Branciari}, \cite{Lakshmikanthamandciric}, \cite{Nieto5}, \cite{Saadati} etc) and references therein.\\
\hspace*{0.5cm}Recently, Samet \emph{et al.} \cite{SametVetro} introduced a new category of contractive type mappings known as $\alpha$-$\psi$ contractive type mapping. The results obtained by Samet \emph{et al.} \cite{SametVetro} extended and generalized the existing fixed point results in the literature, in particular the Banach contraction principle. Further, Karapinar and Samet \cite{KarapinarandSamet} generalized the $\alpha$-$\psi$-contractive type mappings and obtained various fixed point theorems for this generalized class of contractive mappings.\\
\hspace*{0.5cm}The study related to common fixed points of mappings satisfying certain contractive conditions has been at the center of vigorous research activity.  In this paper, some coincidence and common fixed point theorems are obtained for the generalized $\alpha$-$\psi$ contractive pair of mappings. Our results unify and generalize the results derived by Karapinar and Samet \cite{KarapinarandSamet}, Samet \emph{et al.} \cite{SametVetro}, $\acute{C}$iri$\acute{c}$ et al. \cite{CiricMonotone}  and various other related results in the literature. Moreover, from our main results, we will derive various common fixed point results for metric spaces endowed with a partial order and that for cyclic contractive mappings. The presented results extend and generalize numerous related results in the literature.

\section{Preliminaries}
\indent First we introduce some notations and definitions that will be used subsequently.

\begin{defn}(See \cite{SametVetro}).
Let $\Psi$ be the family of functions $\psi : [0, \infty) \rightarrow [0, \infty)$ satisfying the following conditions:\\
(i) $\psi$ is nondecreasing.\\
(ii) $\displaystyle \sum_{n=1}^{+\infty} \psi^{n}(t) < \infty$ for all $t > 0$, where $\psi^{n}$ is the $n^{th}$ iterate of $\psi$.\\
\end{defn}

These functions are known as (c)-comparison functions in the literature. It can be easily verified that if $\psi$ is a (c)-comparison function, then $\psi(t) < t$ for any $t > 0$.\\

\noindent Recently, Samet \emph{et al.} \cite{SametVetro} introduced the following new notions of $\alpha$-$\psi$-contractive type mappings and $\alpha$-admissible mappings:
\begin{defn}
Let $(X, d)$ be a metric space and $T : X \rightarrow X$ be a given self mapping. $T$ is said to be an $\alpha$-$\psi$-contractive mapping if there exists two functions $\alpha : X \times X \rightarrow [0, +\infty)$ and $\psi \in \Psi$ such that
\begin{center}
$\alpha(x, y)d(Tx, Ty) \leq \psi(d(x, y))$
\end{center}
for all $x, y \in X$.
\end{defn}

\begin{defn}
Let $T : X \rightarrow X$ and $\alpha : X \times X \rightarrow [0, +\infty)$. $T$ is said to be $\alpha$-admissible if
\begin{center}
$x, y \in X$, $\alpha(x, y) \geq 1 \Rightarrow \alpha(Tx, Ty) \geq 1$.
\end{center}
\end{defn}

The following fixed point theorems are the main results in \cite{SametVetro}:

\begin{thm}\label{1.1}
Let $(X, d)$ be a complete metric space and $T : X \rightarrow X$ be an $\alpha$-$\psi$-contractive mapping satisfying the following conditions: \\
(i) $T$ is $\alpha$-admissible;\\
(ii) there exists $x_{0} \in X$ such that $\alpha(x_{0}, Tx_{0}) \geq 1$;\\
(iii) $T$ is continuous.\\
\noindent Then, $T$ has a fixed point, that is, there exists $x^{*} \in X$ such that $Tx^{*} = x^{*}$.
\end{thm}

\begin{thm}\label{1.2}
Let $(X, d)$ be a complete metric space and $T : X \rightarrow X$ be an $\alpha$-$\psi$-contractive mapping satisfying the following conditions: \\
(i) $T$ is $\alpha$-admissible;\\
(ii) there exists $x_{0} \in X$ such that $\alpha(x_{0}, Tx_{0}) \geq 1$;\\
(iii) if $\{x_{n}\}$ is a sequence in $X$ such that $\alpha(x_{n}, x_{n+1}) \geq 1$ for all $n$ and $x_{n} \rightarrow x \in X$ as $n \rightarrow + \infty$, then $\alpha(x_{n}, x) \geq 1$ for all $n$.\\
\noindent Then, $T$ has a fixed point.
\end{thm}

Samet \emph{et al.} \cite{SametVetro} added the following condition to the hypotheses of Theorem \ref{1.1} and Theorem \ref{1.2} to assure the uniqueness of the fixed point:\\
(C): For all $x, y \in X$, there exists $z \in X$ such that $\alpha(x, z) \geq 1$ and $\alpha(y, z) \geq 1$.\\

Recently, Karapinar and Samet \cite{KarapinarandSamet} introduced the following concept of generalized $\alpha$-$\psi$-contractive type mappings:

\begin{defn}
Let $(X, d)$ be a metric space and $T : X \rightarrow X$ be a given mapping. We say that $T$ is a generalized $\alpha$-$\psi$-contractive type mapping if there exists two functions $\alpha : X \times X \rightarrow [0, \infty)$ and $\psi \in \Psi$ such that for all $x, y \in X$, we have
\begin{eqnarray*}
\displaystyle \alpha(x, y)d(Tx, Ty) \leq \psi(M(x, y)),
\end{eqnarray*}
where $\displaystyle M(x, y) = \max \left\{d(x, y), \frac{d(x, Tx) + d(y, Ty)}{2}, \frac{d(x, Ty) + d(y, Tx)}{2} \right\}$.
\end{defn}

Further, Karapinar and Samet \cite{KarapinarandSamet} established fixed point theorems for this new class of contractive mappings. Also, they obtained fixed point theorems on metric spaces endowed with a partial order and fixed point theorems for cyclic contractive mappings.

\begin{defn}\cite{AydiCone}
Let $X$ be a non-empty set, $N$ is a natural number such that $N \geq 2$ and $T_{1}, T_{2}, . . ., T_{N} : X \rightarrow X$ are given self-mappings on $X$. If $w = T_{1}x = T_{2}x = . . . = T_{N}x$ for some $x \in X$, then $x$ is called a coincidence point of $T_{1}, T_{2}, . . . , T_{N-1}$ and $T_{N}$, and $w$ is called a point of coincidence of $T_{1}, T_{2}, . . ., T_{N-1}$ and $T_{N}$. If $w = x$, then $x$ is called a common fixed point of $T_{1}, T_{2}, . . . , T_{N-1}$ and $T_{N}$.
\end{defn}

\noindent Let $f, g : X \rightarrow X$ be two mappings. We denote by $C(g, f)$ the set of coincidence points of $g$ and $f$; that is,
\begin{eqnarray*}
C(g, f) = \{z \in X : gz = fz\}
\end{eqnarray*}

\section{Main results}
We start the main section by introducing the new consepts of $\alpha$-admissible w.r.t $g$ mapping and generalized $\alpha$-$\psi$ contractive pair of mappings.

\begin{defn}
Let $f, g : X \rightarrow X$ and $\alpha : X \times X \rightarrow [0, \infty)$. We say that $f$ is $\alpha$-admissible w.r.t $g$ if for all $x, y \in X$, we have
\begin{center}
$\alpha(gx, gy) \geq 1 \Rightarrow \alpha(fx, fy) \geq 1$.
\end{center}
\end{defn}

\begin{rem}
Clearly, every $\alpha$-admissible mapping is $\alpha$-admissible w.r.t $g$ mapping when $g = I$.
\end{rem}

The following example shows that a mapping which is $\alpha$-admissible w.r.t $g$ may not be $\alpha$-admissible.

\begin{exam}
Let $X = [1, \infty)$. Define the mapping $\alpha : X \times X \rightarrow [0, \infty)$ by
\begin{center}
$\alpha(x, y) =
\left\{
     \begin{array}{ll}
       2 & \mbox{if }  x > y  \\
       \displaystyle \frac{1}{3} & otherwise
     \end{array}
\right.$
\end{center}
Also, define the mappings $f, g : X \rightarrow X$ by $f(x) = \displaystyle \frac{1}{x}$ and $g(x) = e^{-x}$ for all $x \in X$.\\
Suppose that $\alpha(x, y) \geq 1$. This implies from the definition of $\alpha$ that $x > y$ which further implies that $\displaystyle \frac{1}{x} < \frac{1}{y}$. Thus, $\alpha(fx, fy)\ngeq 1$, that is, $f$ is not $\alpha$-admissible. \\
Now, we prove that $f$ is $\alpha$-admissible w.r.t $g$. Let us suppose that $\alpha(gx, gy) \geq 1$. So,
\begin{center}
$\alpha(gx, gy) \geq 1 \Rightarrow gx > gy \Rightarrow e^{-x} > e^{-y} \Rightarrow \displaystyle \frac{1}{x} > \frac{1}{y} \Rightarrow \alpha(fx, fy) \geq 1$
\end{center}
Therefore, $f$ is $\alpha$-admissible w.r.t $g$.
\end{exam}

In what follows, we present examples of $\alpha$-admissible w.r.t $g$ mappings.
\begin{exam}
Let $X$ be the set of all non-negative real numbers. Let us define the mapping $\alpha : X \times X \rightarrow [0, +\infty)$ by
\begin{center}
$\alpha(x, y) = \left\{
\begin{array}{c l}
  1 & if \hspace*{0.2cm} x \geq y,\\
  0 & if \hspace*{0.2cm} x < y.\\
\end{array}
\right.$ \\
\end{center}
and define the mappings $f, g : X \rightarrow X$ by $f(x) = e^{x}$ and $g(x) = x^{2}$ for all $x \in X$. Thus, the mapping $f$ is $\alpha$-admissible w.r.t $g$.
\end{exam}

\begin{exam}
Let $X = [1, \infty)$. Let us define the mapping $\alpha : X \times X \rightarrow [0, +\infty)$ by
\begin{center}
$\alpha(x, y) = \left\{
\begin{array}{c l}
  3 & if \hspace*{0.2cm} x, y \in [0, 1],\\
  \frac{1}{2} & \hspace*{0.2cm} otherwise .\\
\end{array}
\right.$ \\
\end{center}
and define the mappings $f, g : X \rightarrow X$ by $f(x) = ln \left(1 + \displaystyle \frac{x}{3}\right)$ and $g(x) = \sqrt{x}$ for all $x \in X$. Thus, the mapping $f$ is $\alpha$-admissible w.r.t $g$.
\end{exam}

Next, we present the new notion of generalized $\alpha$-$\psi$ contractive pair of mappings as follows:

\begin{defn}
Let $(X, d)$ be a metric space and $f, g : X \rightarrow X$ be given mappings. We say that the pair $(f, g)$ is a generalized $\alpha$-$\psi$ contractive pair of mappings if there exists two functions $\alpha : X \times X \rightarrow [0, +\infty)$ and $\psi \in \Psi$ such that for all $x, y \in X$, we have
\begin{eqnarray}
\alpha(gx, gy)d(fx, fy) \leq \psi(M(gx, gy)),
\end{eqnarray}
where $\displaystyle M(gx, gy) = \max \left\{d(gx, gy), \frac{d(gx, fx) + d(gy, fy)}{2}, \frac{d(gx, fy) + d(gy, fx)}{2} \right\}$.
\end{defn}

Our first result is the following coincidence point theorem.
\begin{thm}\label{main}
Let $(X, d)$ be a complete metric space and $f, g : X \rightarrow X$ be such that $f(X) \subseteq g(X)$. Assume that the pair $(f, g)$ is a generalized $\alpha$-$\psi$ contractive pair of mappings and the following conditions hold:\\
(i) $f$ is $\alpha$-admissible w.r.t. $g$;\\
(ii) there exists $x_{0} \in X$ such that $\alpha(gx_{0}, fx_{0}) \geq 1$;\\
(iii) If $\{gx_{n}\}$ is a sequence in $X$ such that $\alpha(gx_{n}, gx_{n+1}) \geq 1$ for all $n$ and $gx_{n} \rightarrow gz \in g(X)$ as $n \rightarrow \infty$, then there exists a subsequence $\{gx_{n(k)}\}$ of $\{gx_{n}\}$ such that $\alpha(gx_{n(k)}, gz) \geq 1$ for all $k$.  \\
\noindent \hspace*{0.5cm}Also suppose $g(X)$ is closed. Then, $f$ and $g$ have a coincidence point.
\end{thm}
\begin{proof}
In view of condition (ii), let $x_{0} \in X$ be such that $\alpha(gx_{0}, fx_{0}) \geq 1$. Since $f(X) \subseteq g(X)$, we can choose a point $x_{1} \in X$ such that $fx_{0} = gx_{1}$. Continuing this process having chosen $x_{1}, x_{2}, . . . , x_{n}$, we choose $x_{n+1}$ in $X$ such that
\begin{eqnarray}
fx_{n} = gx_{n+1}, n = 0, 1, 2, . . .
\end{eqnarray}
Since $f$ is $\alpha$-admissible w.r.t $g$, we have
\begin{center}
$\alpha(gx_{0}, fx_{0}) = \alpha(gx_{0}, gx_{1}) \geq 1 \Rightarrow \alpha(fx_{0}, fx_{1}) = \alpha(gx_{1}, gx_{2}) \geq 1$
\end{center}
Using mathematical induction, we get
\begin{eqnarray}
\alpha(gx_{n}, gx_{n+1}) \geq 1, \forall \hspace*{0.1cm}n = 0, 1, 2, . . .
\end{eqnarray}
If $fx_{n+1} = fx_{n}$ for some $n$, then by (2),
\begin{eqnarray*}
fx_{n+1} = gx_{n+1}, n = 0, 1, 2, . . .
\end{eqnarray*}
that is, $f$ and $g$ have a coincidence point at $x = x_{n+1}$, and so we have finished the proof. For this, we suppose that $d(fx_{n}, fx_{n+1}) > 0$ for all $n$. Applying the inequality (1) and using (3), we obtain
\begin{eqnarray}
d(fx_{n}, fx_{n+1}) &\leq& \alpha(gx_{n}, gx_{n+1})d(fx_{n}, fx_{n+1}) \nonumber \\
&\leq& \psi(M(gx_{n}, gx_{n+1}))
\end{eqnarray}
On the other hand, we have
\begin{eqnarray*}
M(gx_{n}, gx_{n+1}) &=& \max \left\{d(gx_{n}, gx_{n+1}), \displaystyle \frac {d(gx_{n}, fx_{n}) + d(gx_{n+1}, fx_{n+1})}{2}, \displaystyle \frac {d(gx_{n}, fx_{n+1}) + d(gx_{n+1}, fx_{n})}{2} \right\} \nonumber \\
&\leq& \max\{d(fx_{n-1}, fx_{n}), d(fx_{n}, fx_{n+1})\}
\end{eqnarray*}
Owing to monotonicity of the function $\psi$ and using the inequalities (2) and (4), we have for all $n \geq 1$
\begin{eqnarray}
d(fx_{n}, fx_{n+1}) &\leq& \psi(\max \left\{d(fx_{n-1}, fx_{n}), d(fx_{n}, fx_{n+1}) \right\}
\end{eqnarray}
If for some $n \geq 1$, we have $d(fx_{n-1}, fx_{n}) \leq d(fx_{n}, fx_{n+1})$, from (5), we obtain that
\begin{eqnarray*}
\hspace*{2.0cm}d(fx_{n}, fx_{n+1}) &\leq& \psi(d(fx_{n}, fx_{n+1})) < d(fx_{n}, fx_{n+1}),
\end{eqnarray*}
a contradiction.  Thus, for all $n \geq 1$, we have
\begin{eqnarray}
\max \left\{d(fx_{n-1}, fx_{n}), d(fx_{n}, fx_{n+1}) \right\} = d(fx_{n-1}, fx_{n})
\end{eqnarray}
Notice that in view of (5) and (6), we get for all $n \geq 1$ that
\begin{eqnarray}
d(fx_{n}, fx_{n+1}) &\leq& \psi(d(fx_{n-1}, fx_{n})).
\end{eqnarray}
Continuing this process inductively, we obtain
\begin{eqnarray}
d(fx_{n}, fx_{n+1}) &\leq& \psi^{n}(d(fx_{0}, fx_{1})), \hspace*{0.5cm} \forall n \geq 1.
\end{eqnarray}
From (8) and using the triangular inequality, for all $k \geq 1$, we have
\begin{eqnarray}
d(fx_{n}, fx_{n+k}) &\leq& d(fx_{n}, fx_{n+1}) + . . . + d(fx_{n+k-1}, fx_{n+k}) \nonumber \\
&\leq& \sum_{p=n}^{n+k-1} \psi^{p}(d(fx_{1}, fx_{0})) \nonumber \\
&\leq& \sum_{p=n}^{+\infty} \psi^{p}(d(fx_{1}, fx_{0}))
\end{eqnarray}
Letting $p \rightarrow \infty$ in (9), we obtain that $\{fx_{n}\}$ is a Cauchy sequence in $(X, d)$. Since by (2) we have $\{fx_{n}\} = \{gx_{n+1}\} \subseteq g(X)$ and $g(X)$ is closed, there exists $z \in X$ such that
\begin{eqnarray}
\displaystyle \lim_{n \rightarrow \infty} gx_{n} = gz.
\end{eqnarray}
Now, we show that $z$ is a coincidence point of $f$ and $g$. On contrary, assume that $d(fz, gz) > 0$. Since by condition (iii) and (10), we have $\alpha(gx_{n(k)}, gz) \geq 1$ for all $k$, then by the use of triangle inequality and (1) we obtain
\begin{eqnarray}\nonumber
d(gz, fz) &\leq& d(gz, fx_{n(k)}) + d(fx_{n(k)}, fz) \\ \nonumber
&\leq& d(gz, fx_{n(k)}) + \alpha(gx_{n(k)}, gz)d(fx_{n(k)}, fz) \\
&\leq& d(gz, fx_{n(k)}) + \psi(M(gx_{n(k)}, gz)
\end{eqnarray}
On the other hand, we have
\begin{eqnarray*}
M(gx_{n(k)}, gz) &=& \max\left\{d(gx_{n(k)}, gz), \frac{d(gx_{n(k)}, fx_{n(k)}) + d(gz, fz)}{2}, \frac{d(gx_{n(k)}, fz) + d(gz, fx_{n(k)})}{2} \right\}
\end{eqnarray*}
Owing to above equality, we get from (11),
\begin{eqnarray}\nonumber
d(gz, fz) &\leq& d(gz, fx_{n(k)}) + \psi(M(gx_{n(k)}, gz) \\ \nonumber
&\leq& d(gz, fx_{n(k)}) + \\&& \nonumber
\psi \left(\max\left\{d(gx_{n(k)}, gz), \frac{d(gx_{n(k)}, fx_{n(k)}) + d(gz, fz)}{2}, \frac{d(gx_{n(k)}, fz) + d(gz, fx_{n(k)})}{2} \right\} \right)
\end{eqnarray}
Letting $k \rightarrow \infty$ in the above inequality yields $\displaystyle d(gz, fz) \leq \psi\left(\frac{d(fz, gz)}{2}\right) < \frac{d(fz, gz)}{2}$, which is a contradiction. Hence, our supposition is wrong and $d(fz, gz) = 0$, that is, $fz = gz$. This shows that $f$ and $g$ have a coincidence point.
\end{proof}

The next theorem shows that under additional hypotheses we can deduce the existence and uniqueness of a common fixed point.

\begin{thm}\label{main2}
In addition to the hypotheses of Theorem \ref{main}, suppose that for all $u, v \in C(g, f)$, there exists $w \in X$ such that $\alpha(gu, gw) \geq 1$ and $\alpha(gv, gw) \geq 1$ and $f, g$ commute at their coincidence points. Then $f$ and $g$ have a unique common fixed point.
\end{thm}
\begin{proof}
We need to consider three steps:\\
Step 1. We claim that if $u, v \in C(g, f)$, then $gu = gv$. By hypotheses, there exists $w \in X$ such that
\begin{eqnarray}
\alpha(gu, gw) \geq 1, \alpha(gv, gw) \geq 1
\end{eqnarray}
Due to the fact that $f(X) \subseteq g(X)$, let us define the sequence $\{w_{n}\}$ in $X$ by $gw_{n+1} = fw_{n}$ for all $n \geq 0$ and $w_{0} = w$. Since $f$ is $\alpha$-admissible w.r.t $g$, we have from (12) that
\begin{eqnarray}
\alpha(gu, gw_{n}) \geq 1, \alpha(gv, gw_{n}) \geq 1
\end{eqnarray}
for all $n \geq 0$. Applying inequality (1) and using (13), we obtain
\begin{eqnarray}
d(gu, gw_{n+1}) &=& d(fu, fw_{n}) \nonumber\\
&\leq& \alpha(gu, gw_{n})d(fu, fw_{n}) \nonumber\\
&\leq& \psi(M(gu, gw_{n}))
\end{eqnarray}
On the other hand, we have
\begin{eqnarray}
M(gu, gw_{n}) &=& \max \left\{d(gu, gw_{n}), \frac{d(gu, fu) + d(gw_{n}, fw_{n})}{2}, \frac{d(gu, fw_{n}) + d(gw_{n}, fu)}{2} \right\} \nonumber \\
&\leq& \max \left\{d(gu, gw_{n}), d(gu, gw_{n+1}) \right\}
\end{eqnarray}
Using the above inequality, (14) and owing to the monotone property of $\psi$, we get that
\begin{eqnarray}
d(gu, gw_{n+1}) &\leq& \psi(\max \left\{d(gu, gw_{n}), d(gu, gw_{n+1}) \right\})
\end{eqnarray}
for all $n$. Without restriction to the generality, we can suppose that $d(gu, gw_{n}) > 0$ for all $n$. If $\max\{d(gu, gw_{n}), d(gu, gw_{n+1})\} = d(gu, gw_{n+1})$, we have from (16) that
\begin{eqnarray}
d(gu, gw_{n+1}) \leq \psi(d(gu, gw_{n+1})) < d(gu, gw_{n+1})
\end{eqnarray}
which is a contradiction. Thus, we have $\max \{d(gu, gw_{n}), d(gu, gw_{n+1}) \} = d(gu, gw_{n})$, and $d(gu, gw_{n+1}) \leq \psi(d(gu, gw_{n}))$ for all $n$. This implies that
\begin{eqnarray}
d(gu, gw_{n}) \leq \psi^{n}(d(gu, gw_{0})), \hspace*{0.1cm} \forall n \geq 1
\end{eqnarray}
Letting $n \rightarrow \infty$ in the above inequality, we infer that
\begin{eqnarray}
\lim_{n \rightarrow \infty}d(gu, gw_{n}) = 0
\end{eqnarray}
Similarly, we can prove that
\begin{eqnarray}
\lim_{n \rightarrow \infty}d(gv, gw_{n}) = 0
\end{eqnarray}
It follows from (19) and (20) that $gu = gv$.\\
Step 2. Existence of a common fixed point:
Let $u \in C(g, f)$, that is, $gu = fu$. Owing to the commutativity of $f$ and $g$ at their coincidence points, we get
\begin{eqnarray}
g^{2}u = gfu = fgu
\end{eqnarray}
Let us denote $gu = z$, then from (21), $gz = fz$. Thus, $z$ is a coincidence point of $f$ and $g$. Now, from Step 1, we have $gu = gz = z =
fz$. Then, $z$ is a common fixed point of $f$ and $g$.\\
Step 3. Uniqueness: Assume that $z^{*}$ is another common fixed point of $f$ and $g$. Then $z^{*} \in C(g, f)$. By step 1, we have $z^{*} = gz^{*} = gz = z$. This completes the proof.
\end{proof}

In what follows, we furnish an illustrative example wherein one demonstrates Theorem \ref{main2} on the existence and uniqueness of a common fixed point.

\begin{exam}
Consider $X = [0, +\infty)$ equipped with the usual metric $d(x, y) = |x - y|$ for all $x, y \in X$. Define the mappings $f : X \rightarrow X$ and $g : X \rightarrow X$ by
\begin{center}
$f(x) = \left\{
\begin{array}{c l}
  \displaystyle 2x - \displaystyle \frac{3}{2} & if \hspace*{0.1cm} x > 2 ,\\
  \displaystyle \frac{x}{3}                    & if \hspace*{0.1cm} 0 \leq x \leq 2.\\
\end{array}
\right.$ \\
\end{center}
and
\begin{center}
$g(x) = \displaystyle \frac{x}{2} \hspace*{0.1cm} \forall \hspace*{0.01cm} x \in X$.
\end{center}
Now, we define the mapping $\alpha : X \times X \rightarrow [0, +\infty)$ by
\begin{center}
$\alpha(x, y) = \left\{
\begin{array}{c l}
  1 & if \hspace*{0.1cm}(x, y) \in [0, 1] ,\\
  0 & otherwise.\\
\end{array}
\right.$ \\
\end{center}
Clearly, the pair $(f, g)$ is a generalized $\alpha$-$\psi$ contractive pair of mappings with $\psi(t) = \displaystyle \frac{4}{5}t$ for all $t \geq 0$. In fact, for all $x, y \in X$, we have
\begin{eqnarray*}
\alpha(gx, gy). d(fx, fy) = 1. \left |\displaystyle \frac{x}{3} - \displaystyle \frac{y}{3} \right| &\leq& \frac{4}{5} \left |\displaystyle \frac{x}{2} - \displaystyle \frac{y}{2} \right| \\
&=& \displaystyle \frac{4}{5} d(gx, gy) \\
&\leq& \displaystyle \frac{4}{5} M(gx, gy) = \psi(M(gx, gy))
\end{eqnarray*}
Moreover, there exists $x_{0} \in X$ such that $\alpha(gx_{0}, fx_{0}) \geq 1$. Infact, for $x_{0} = 1$, we have $\alpha \left(\displaystyle \frac{1}{2}, \displaystyle \frac{1}{3} \right) = 1$.\\
\noindent Now, it remains to show that $f$ is $\alpha$-admissible w.r.t $g$. In so doing, let $x, y \in X$ such that $\alpha(gx, gy) \geq 1$. This implies that $gx, gy \in [0, 1]$ and by the definition of $g$, we have $x, y \in [0, 2]$. Therefore, by the definition of $f$ and $\alpha$, we have
\begin{center}
$f(x) = \displaystyle \frac{x}{3} \in [0, 1]$, $f(y) = \displaystyle \frac{y}{3} \in [0, 1]$ and $\alpha(fx, fy) = 1$.
\end{center}
Thus, $f$ is $\alpha$-admissible w.r.t $g$. Clearly, $f(X) \subseteq g(X)$ and $g(X)$ is closed. \\
\noindent Finally, let $\{gx_{n}\}$ be a sequence in $X$ such that $\alpha(gx_{n}, gx_{n+1}) \geq 1$ for all $n$ and $gx_{n} \rightarrow gz \in g(X)$ as $n \rightarrow +\infty$. Since $\alpha(gx_{n}, gx_{n+1}) \geq 1$ for all $n$, by the definition of $\alpha$, we have $gx_{n} \in [0, 1]$ for all $n$ and $gz \in [0, 1]$. Then, $\alpha(gx_{n}, gz) \geq 1$.
Now, all the hypotheses of Theorem \ref{main} are satisfied. Consequently, $f$ and $g$ have a coincidence point. Here, 0 is a coincidence point of $f$ and $g$. Also, clearly all the hypotheses of Theorem \ref{main2} are satisfied. In this example, 0 is the unique common fixed point of $f$ and $g$.
\end{exam}

\section{Consequences}
In this section, we will show that many existing results in the literature can be obtained easily from our Theorem \ref{main2}.

\subsection{Standard Fixed Point Theorems}
By taking $\alpha(x, y) = 1$ for all $x, y \in X$ in Theorem \ref{main2}, we obtain immediately the following fixed point theorem.
\begin{cor}\label{cons1}
Let $(X, d)$ be a complete metric space and $f, g : X \rightarrow X$ be such that $f(X) \subseteq g(X)$. Suppose that there exists a function $\psi \in \Psi$ such that
\begin{eqnarray}
d(fx, fy) \leq \psi(M(gx, gy)),
\end{eqnarray}
for all $x, y \in X$.  Also, suppose $g(X)$ is closed. Then, $f$ and $g$ have a coincidence point. Further, if $f$, $g$ commute at their coincidence points, then $f$ and $g$ have a common fixed point.
\end{cor}

By taking $g = I$ in Corollary \ref{cons1}, we obtain immediately the following fixed point theorem.
\begin{cor}(see Karapinar and Samet \cite{KarapinarandSamet}).\label{cons2}
Let $(X, d)$ be a complete metric space and $f : X \rightarrow X$. Suppose that there exists a function $\psi \in \Psi$ such that
\begin{eqnarray}
d(fx, fy) \leq \psi(M(x, y)),
\end{eqnarray}
for all $x, y \in X$. Then $f$ has a unique fixed point.
\end{cor}

The following fixed point theorems can be easily obtained from Corollaries \ref{cons1} and \ref{cons2}.
\begin{cor}
Let $(X, d)$ be a complete metric space and $f, g : X \rightarrow X$ be such that $f(X) \subseteq g(X)$. Suppose that there exists a function $\psi \in \Psi$ such that
\begin{eqnarray}
d(fx, fy) \leq \psi(d(gx, gy)),
\end{eqnarray}
for all $x, y \in X$.  Also, suppose $g(X)$ is closed. Then, $f$ and $g$ have a coincidence point. Further, if $f$, $g$ commute at their coincidence points, then $f$ and $g$ have a common fixed point.
\end{cor}

\begin{cor}(Berinde \cite{Berinde}).
Let $(X, d)$ be a complete metric space and $f : X \rightarrow X$. Suppose that there exists a function $\psi \in \Psi$ such that
\begin{eqnarray}
d(fx, fy) \leq \psi(d(x, y)),
\end{eqnarray}
for all $x, y \in X$. Then $f$ has a unique fixed point.
\end{cor}

\begin{cor}($\acute{C}$iri$\acute{c}$ \cite{Ciric72}).
Let $(X, d)$ be a complete metric space and $T : X \rightarrow X$ be a given mapping. Suppose that there exists a constant $\lambda \in (0, 1)$
such that
\begin{eqnarray*}
d(fx, fy) \leq \lambda \max \left\{d(x, y), \frac{d(x, fx) + d(y, fy)}{2}, \frac{d(x, fy) + d(y, fx)}{2} \right\}
\end{eqnarray*}
for all $x, y \in X$. Then $T$ has a unique fixed point.
\end{cor}

\begin{cor}(Hardy and Rogers \cite{Hardy})
Let $(X, d)$ be a complete metric space and $T : X \rightarrow X$ be a given mapping. Suppose that there exists constants $A, B, C \geq 0$ with $(A + 2B + 2C) \in (0, 1)$ such that
\begin{eqnarray*}
d(fx, fy) \leq Ad(x, y) + B[d(x, fx) + d(y, fy)] + C[d(x, fy) + d(y, fx)],
\end{eqnarray*}
for all $x, y \in X$. Then $T$ has a unique fixed point.
\end{cor}

\begin{cor}(Banach Contraction Principle \cite{Banach})
Let $(X, d)$ be a complete metric space and $T : X \rightarrow X$ be a given mapping. Suppose that there exists a constant $\lambda \in (0, 1)$ such that
\begin{eqnarray*}
d(fx, fy) \leq \lambda d(x, y)
\end{eqnarray*}
for all $x, y \in X$. Then $T$ has a unique fixed point.
\end{cor}

\begin{cor}(Kannan \cite{Kannan})
Let $(X, d)$ be a complete metric space and $T : X \rightarrow X$ be a given mapping. Suppose that there exists a constant $\lambda \in (0, 1/2)$ such that
\begin{eqnarray*}
d(fx, fy) \leq \lambda[d(x, fx) + d(y, fy)],
\end{eqnarray*}
for all $x, y \in X$. Then $T$ has a unique fixed point.
\end{cor}

\begin{cor}(Chatterjee \cite{Chatterjea})
Let $(X, d)$ be a complete metric space and $T : X \rightarrow X$ be a given mapping. Suppose that there exists a constant $\lambda \in (0, 1/2)$ such that
\begin{eqnarray*}
d(fx, fy) \leq \lambda[d(x, fy) + d(y, fx)],
\end{eqnarray*}
for all $x, y \in X$. Then $T$ has a unique fixed point.
\end{cor}

\subsection{Fixed Point Theorems on Metric Spaces Endowed with a Partial Order}
Recently, there have been enormous developments in the study of fixed point problems of contractive mappings in metric spaces endowed with a partial order. The first result in this direction was given by Turinici \cite{Turinici}, where he extended the Banach contraction principle in partially ordered sets. Some applications of Turinici's theorem to matrix equations were presented by Ran and Reurings \cite{Ran}. Later, many useful results have been obtained regarding the existence of a fixed point for contraction type mappings in partially ordered metric spaces by Bhaskar and Lakshmikantham \cite{BhaskarandLakshmikantham}, Nieto and Lopez \cite{Nieto5,Nieto7}, Agarwal et al. \cite{Agarwal}, Lakshmikantham and $\acute{C}$iri$\acute{c}$ \cite{Lakshmikanthamandciric} and Samet \cite{SametMeir} etc. In this section, we will derive various fixed point results on a metric space endowed with a partial order. For this, we require the following concepts:

\begin{defn}\cite{KarapinarandSamet}
Let $(X, \preceq)$ be a partially ordered set and $T : X \rightarrow X$ be a given mapping. We say that $T$ is nondecreasing with respect to $\preceq$ if
\begin{center}
$x, y \in X, x \preceq y \Rightarrow Tx \preceq Ty$.
\end{center}
\end{defn}

\begin{defn}\cite{KarapinarandSamet}
Let $(X, \preceq)$ be a partially ordered set. A sequence $\{x_{n}\} \subset X$ is said to be nondecreasing with respect to $\preceq$ if $x_{n} \preceq x_{n+1}$ for all $n$.
\end{defn}

\begin{defn}\cite{KarapinarandSamet}
Let $(X, \preceq)$ be a partially ordered set and $d$ be a metric on $X$. We say that $(X, \preceq, d)$ is regular if for every nondecreasing sequence $\{x_{n}\} \subset X$ such that $x_{n} \rightarrow x \in X$ as $n \rightarrow \infty$, there exists a subsequence $\{x_{n(k)}\}$ of $\{x_{n}\}$ such that $x_{n(k)} \preceq x$ for all $k$.
\end{defn}

\begin{defn}\cite{CiricMonotone}
Suppose $(X, \preceq)$ is a partially ordered set and $F, g : X \rightarrow X$ are mappings of $X$ into itself. One says $F$ is $g$-non-decreasing if for $x, y \in X$,
\begin{eqnarray}
g(x) \preceq g(y) \hspace*{0.5cm} implies \hspace*{0.5cm} F(x) \preceq F(y).
\end{eqnarray}
\end{defn}

\begin{defn}
Let $(X, \preceq)$ be a partially ordered set and $d$ be a metric on $X$. We say that $(X, \preceq, d)$ is $g$-regular where $g : X \rightarrow X$ if for every nondecreasing sequence $\{gx_{n}\} \subset X$ such that $gx_{n} \rightarrow gz \in X$ as $n \rightarrow \infty$, there exists a subsequence $\{gx_{n(k)}\}$ of $\{gx_{n}\}$ such that $gx_{n(k)} \preceq gz$ for all $k$.
\end{defn}

We have the following result.
\begin{cor}\label{partialorder}
Let $(X, \preceq)$ be a partially ordered set and $d$ be a metric on $X$ such that $(X, d)$ is complete. Assume that $f, g : X \rightarrow X$ be such that $f(X) \subseteq g(X)$ and $f$ be a $g$-non-decreasing mapping w.r.t $\preceq$. Suppose that there exists a function $\psi \in \Psi$ such that
\begin{eqnarray}
d(fx, fy) \leq \psi(M(gx, gy)),
\end{eqnarray}
for all $x, y \in X$ with $gx \preceq gy$. Suppose also that the following conditions hold:\\
(i) there exists $x_{0} \in X$ s.t $gx_{0} \preceq fx_{0}$;\\
(ii) $(X, \preceq, d)$ is $g$-regular.\\
Also suppose $g(X)$ is closed. Then, $f$ and $g$ have a coincidence point. Moreover, if for every pair $(x, y) \in C(g,f) \times C(g, f)$ there exists $z \in X$ such that $gx \preceq gz$ and $gy \preceq gz$, and if $f$ and $g$ commute at their coincidence points, then we obtain uniqueness of the common fixed point.
\end{cor}
\begin{proof}
Define the mapping $\alpha : X \times X \rightarrow [0, \infty)$ by
\begin{eqnarray}
\alpha(x, y) =
\left\{
     \begin{array}{ll}
       1 & \mbox{if }  x \preceq y \hspace*{0.2cm} or \hspace*{0.2cm} x \succeq y  \\
       0 & otherwise
     \end{array}
\right.
\end{eqnarray}
Clearly, the pair $(f, g)$ is a generalized $\alpha$-$\psi$ contractive pair of mappings, that is,
\begin{eqnarray*}
\alpha(gx, gy)d(fx, fy) \leq \psi(M(gx, gy)),
\end{eqnarray*}
for all $x, y \in X$. Notice that in view of condition (i), we have $\alpha(gx_{0}, fx_{0}) \geq 1$. Moreover, for all $x, y \in X$, from the $g$-monotone property of $f$, we have
\begin{eqnarray}
\alpha(gx, gy) \geq 1 \Rightarrow gx \hspace*{0.1cm} \preceq \hspace*{0.1cm} gy \hspace*{0.1cm} or \hspace*{0.1cm} gx \succeq gy \Rightarrow  fx \preceq fy \hspace*{0.1cm} or \hspace*{0.1cm} fx \succeq fy \Rightarrow \alpha(fx, fy) \geq 1.
\end{eqnarray}
which amounts to say that $f$ is $\alpha$-admissible w.r.t $g$. Now, let $\{gx_{n}\}$ be a sequence in $X$ such that $\alpha(gx_{n}, gx_{n+1}) \geq 1$ for all $n$ and $gx_{n} \rightarrow gz \in X$ as $n \rightarrow \infty$. From the $g$-regularity hypothesis, there exists a subsequence $\{gx_{n(k)}\}$ of $\{gx_{n}\}$ such that $gx_{n(k)} \preceq gz$ for all $k$. So, by the definition of $\alpha$, we obtain that $\alpha(gx_{n(k)}, gz) \geq 1$. Now, all the hypotheses of Theorem \ref{main} are satisfied. Hence, we deduce that $f$ and $g$ have a coincidence point $z$, that is, $fz = gz$.\\
Now, we need to show the existence and uniqueness of common fixed point. For this, let $x, y \in X$. By hypotheses, there exists $z \in X$ such that $gx \preceq gz$ and $gy \preceq gz$, which implies from the definition of $\alpha$ that $\alpha(gx, gz) \geq 1$ and $\alpha(gy, gz) \geq 1$. Thus, we deduce the existence and uniqueness of the common fixed point by Theorem \ref{main2}.
\end{proof}

The following results are immediate consequences of Corollary \ref{partialorder}.

\begin{cor}
Let $(X, \preceq)$ be a partially ordered set and $d$ be a metric on $X$ such that $(X, d)$ is complete. Assume that $f, g : X \rightarrow X$ and $f$ be a $g$-non-decreasing mapping w.r.t $\preceq$. Suppose that there exists a function $\psi \in \Psi$ such that
\begin{eqnarray}
d(fx, fy) \leq \psi(d(gx, gy)),
\end{eqnarray}
for all $x, y \in X$ with $gx \preceq gy$. Suppose also that the following conditions hold:\\
(i) there exists $x_{0} \in X$ s.t $gx_{0} \preceq fx_{0}$;\\
(ii) $(X, \preceq, d)$ is $g$-regular.\\
Also suppose $g(X)$ is closed. Then, $f$ and $g$ have a coincidence point. Moreover, if for every pair $(x, y) \in C(g,f) \times C(g, f)$ there exists $z \in X$ such that $gx \preceq gz$ and $gy \preceq gz$, and if $f$ and $g$ commute at their coincidence points, then we obtain uniqueness of the common fixed point.
\end{cor}

\begin{cor}
Let $(X, \preceq)$ be a partially ordered set and $d$ be a metric on $X$ such that $(X, d)$ is complete. Assume that $f, g : X \rightarrow X$ and $f$ be a $g$-non-decreasing mapping w.r.t $\preceq$. Suppose that there exists a constant $\lambda \in (0, 1)$ such that
\begin{eqnarray}
d(fx, fy) \leq \lambda \max \left\{d(gx, gy), \frac{d(gx, fx) + d(gy, fy)}{2}, \frac{d(gx, fy) + d(gy, fx)}{2} \right\},
\end{eqnarray}
for all $x, y \in X$ with $gx \preceq gy$. Suppose also that the following conditions hold:\\
(i) there exists $x_{0} \in X$ s.t $gx_{0} \preceq fx_{0}$;\\
(ii) $(X, \preceq, d)$ is $g$-regular.\\
Also suppose $g(X)$ is closed. Then, $f$ and $g$ have a coincidence point. Moreover, if for every pair $(x, y) \in C(g,f) \times C(g, f)$ there exists $z \in X$ such that $gx \preceq gz$ and $gy \preceq gz$, and if $f$ and $g$ commute at their coincidence points, then we obtain uniqueness of the common fixed point.
\end{cor}

\begin{cor}
Let $(X, \preceq)$ be a partially ordered set and $d$ be a metric on $X$ such that $(X, d)$ is complete. Assume that $f, g : X \rightarrow X$ and $f$ be a $g$-non-decreasing mapping w.r.t $\preceq$. Suppose that there exists constants $A, B, C \geq 0$ with $(A+2B+2C) \in (0, 1)$ such that
\begin{eqnarray}
d(fx, fy) \leq Ad(gx, gy) + B[d(gx, fx) + d(gy, fy)] + C[d(gx, fy) + d(gy, fx)],
\end{eqnarray}
for all $x, y \in X$ with $gx \preceq gy$. Suppose also that the following conditions hold:\\
(i) there exists $x_{0} \in X$ s.t $gx_{0} \preceq fx_{0}$;\\
(ii) $(X, \preceq, d)$ is $g$-regular.\\
Also suppose $g(X)$ is closed. Then, $f$ and $g$ have a coincidence point. Moreover, if for every pair $(x, y) \in C(g,f) \times C(g, f)$ there exists $z \in X$ such that $gx \preceq gz$ and $gy \preceq gz$, and if $f$ and $g$ commute at their coincidence points, then we obtain uniqueness of the common fixed point.
\end{cor}

\begin{cor}
Let $(X, \preceq)$ be a partially ordered set and $d$ be a metric on $X$ such that $(X, d)$ is complete. Assume that $f, g : X \rightarrow X$ and $f$ be a $g$-non-decreasing mapping w.r.t $\preceq$. Suppose that there exists a constant $\lambda \in (0, 1)$ such that
\begin{eqnarray}
d(fx, fy) \leq \lambda(d(gx, gy)),
\end{eqnarray}
for all $x, y \in X$ with $gx \preceq gy$. Suppose also that the following conditions hold:\\
(i) there exists $x_{0} \in X$ s.t $gx_{0} \preceq fx_{0}$;\\
(ii) $(X, \preceq, d)$ is $g$-regular.\\
Also suppose $g(X)$ is closed. Then, $f$ and $g$ have a coincidence point. Moreover, if for every pair $(x, y) \in C(g,f) \times C(g, f)$ there exists $z \in X$ such that $gx \preceq gz$ and $gy \preceq gz$, and if $f$ and $g$ commute at their coincidence points, then we obtain uniqueness of the common fixed point.
\end{cor}

\begin{cor}
Let $(X, \preceq)$ be a partially ordered set and $d$ be a metric on $X$ such that $(X, d)$ is complete. Assume that $f, g : X \rightarrow X$ and $f$ be a $g$-non-decreasing mapping w.r.t $\preceq$. Suppose that there exists constants $A, B, C \geq 0$ with $(A+2B+2C) \in (0, 1)$ such that
\begin{eqnarray}
d(fx, fy) \leq \lambda[d(gx, fx) + d(gy, fy)],
\end{eqnarray}
for all $x, y \in X$ with $gx \preceq gy$. Suppose also that the following conditions hold:\\
(i) there exists $x_{0} \in X$ s.t $gx_{0} \preceq fx_{0}$;\\
(ii) $(X, \preceq, d)$ is $g$-regular.\\
Also suppose $g(X)$ is closed. Then, $f$ and $g$ have a coincidence point. Moreover, if for every pair $(x, y) \in C(g,f) \times C(g, f)$ there exists $z \in X$ such that $gx \preceq gz$ and $gy \preceq gz$, and if $f$ and $g$ commute at their coincidence points, then we obtain uniqueness of the common fixed point.
\end{cor}

\begin{cor}
Let $(X, \preceq)$ be a partially ordered set and $d$ be a metric on $X$ such that $(X, d)$ is complete. Assume that $f, g : X \rightarrow X$ and $f$ be a $g$-non-decreasing mapping w.r.t $\preceq$. Suppose that there exists constants $A, B, C \geq 0$ with $(A+2B+2C) \in (0, 1)$ such that
\begin{eqnarray}
d(fx, fy) \leq \lambda[d(gx, fy) + d(gy, fx)],
\end{eqnarray}
for all $x, y \in X$ with $gx \preceq gy$. Suppose also that the following conditions hold:\\
(i) there exists $x_{0} \in X$ s.t $gx_{0} \preceq fx_{0}$;\\
(ii) $(X, \preceq, d)$ is $g$-regular.\\
Also suppose $g(X)$ is closed. Then, $f$ and $g$ have a coincidence point. Moreover, if for every pair $(x, y) \in C(g,f) \times C(g, f)$ there exists $z \in X$ such that $gx \preceq gz$ and $gy \preceq gz$, and if $f$ and $g$ commute at their coincidence points, then we obtain uniqueness of the common fixed point.
\end{cor}

\textbf{Remarks}
\begin{itemize}
            \item Letting $g = I_{X}$ in Corollary 4.11 we obtain Corollary 3.12 in \cite{KarapinarandSamet}.
            \item Letting $g = I_{X}$ in Corollary 4.12 we obtain Corollary 3.13 in \cite{KarapinarandSamet}.
            \item Letting $g = I_{X}$ in Corollary 4.13 we obtain Corollary 3.14 in \cite{KarapinarandSamet}.
            \item Letting $g = I_{X}$ in Corollary 4.14 we obtain Corollary 3.15 in \cite{KarapinarandSamet}.
            \item Letting $g = I_{X}$ in Corollary 4.15 we obtain Corollary 3.16 in \cite{KarapinarandSamet}.
            \item Letting $g = I_{X}$ in Corollary 4.16 we obtain Corollary 3.17 in \cite{KarapinarandSamet}.
            \end{itemize}

\subsection{Fixed Point Theorems for Cyclic Contractive Mappings}
As a generalization of the Banach contraction mapping principle, Kirk \emph{et al.} \cite{Kirkcyclic} in 2003 introduced cyclic representations and cyclic contractions. A mapping $T : A \cup B \rightarrow A \cup B$ is called cyclic if $T(A) \subseteq B$ and $T(B) \subseteq A$, where $A, B$ are nonempty subsets of a metric space $(X, d)$. Moreover, $T$ is called a cyclic contraction if there exists $k \in (0, 1)$ such that $d(Tx, Ty) \leq kd(x, y)$ for all $x \in A$ and $y \in B$. Notice that although a contraction is continuous, cyclic contractions need not be. This is one of the important gains of this theorem. In the last decade, several authors have used the cyclic representations and cyclic contractions to obtain various fixed point results. see for example (\cite{Agarwalcyclic,Karapinarcyclic1,Karapinarcyclic2,Pacurar,Petric,RusCyclic}).

\begin{cor}\label{cyclic}
Let $(X, d)$ be a complete metric space, $A_{1}$ and $A_{2}$ are two nonempty closed subsets of $X$ and $f, g : Y \rightarrow Y$ be two mappings, where $Y = A_{1} \cup A_{2}$. Suppose that the following conditions hold:\\
(i) $g(A_{1})$ and $g(A_{2})$ are closed;\\
(ii) $f(A_{1}) \subseteq g(A_{2})$ and $f(A_{2}) \subseteq g(A_{1})$;\\
(iii) $g$ is one-to-one;\\
(iv) there exists a function $\psi \in \Psi$ such that
\begin{eqnarray}
d(fx, fy) \leq \psi(M(gx, gy)), \hspace*{0.1cm} \forall (x, y) \in A_{1} \times A_{2}.
\end{eqnarray}
Then, $f$ and $g$ have a coincidence point $z \in A_{1} \cap A_{2}$. Further, if $f$, $g$ commute at their coincidence points, then $f$ and $g$ have a unique common fixed point that belongs to $A_{1} \cap A_{2}$.
\end{cor}
\begin{proof}
Due to the fact that $g$ is one-to-one, condition (iv) is equivalent to
\begin{eqnarray}
d(fx, fy) \leq \psi(M(gx, gy)), \hspace*{0.1cm} \forall (gx, gy) \in g(A_{1}) \times g(A_{2}).
\end{eqnarray}
Now, since $A_{1}$ and $A_{2}$ are closed subsets of the complete metric space $(X, d)$, then $(Y, d)$ is complete. \\
Define the mapping $\alpha : Y \times Y \rightarrow [0, \infty)$ by
\begin{eqnarray}
\alpha(x, y) =
\left\{
     \begin{array}{ll}
       1 & if \hspace*{0.1cm} (x, y) \in (g(A_{1}) \times g(A_{2})) \cup (g(A_{2}) \times g(A_{1}))\\
       0 & otherwise
     \end{array}
\right.
\end{eqnarray}
Notice that in view of definition of $\alpha$ and condition (iv), we can write
\begin{eqnarray}
\alpha(gx, gy) d(fx, fy) \leq \psi(M(gx, gy))
\end{eqnarray}
for all $gx \in g(A_{1})$ and $gy \in g(A_{2})$. Thus, the pair $(f, g)$ is
a generalized $\alpha$-$\psi$ contractive pair of mappings. \\
By using condition (ii), we can show that $f(Y) \subseteq g(Y)$. Moreover, $g(Y)$ is closed.\\
Next, we proceed to show that $f$ is $\alpha$-admissible w.r.t $g$. Let $(gx, gy) \in Y \times Y$ such that $\alpha(gx, gy) \geq 1$; that is,
\begin{eqnarray}
(gx, gy) \in (g(A_{1}) \times g(A_{2})) \cup (g(A_{2}) \times g(A_{1}))
\end{eqnarray}
Since $g$ is one-to-one, this implies that
\begin{eqnarray}
(x, y) \in (A_{1} \times A_{2}) \cup (A_{2} \times A_{1})
\end{eqnarray}
So, from condition (ii), we infer that
\begin{eqnarray}
(fx, fy) \in (g(A_{2}) \times g(A_{1})) \cup (g(A_{1}) \times g(A_{2}))
\end{eqnarray}
that is, $\alpha(fx, fy) \geq 1$. This implies that $f$ is $\alpha$-admissible w.r.t $g$.\\
Now, let $\{gx_{n}\}$ be a sequence in $X$ such that $\alpha(gx_{n}, gx_{n+1}) \geq 1$ for all $n$ and $gx_{n} \rightarrow gz \in g(X)$ as $n \rightarrow \infty$. From the definition of $\alpha$, we infer that
\begin{eqnarray}
(gx_{n}, gx_{n+1}) \in (gA_{1} \times gA_{2}) \cup (gA_{2} \times gA_{1})
\end{eqnarray}
Since $(gA_{1} \times gA_{2}) \cup (gA_{2} \times gA_{1})$ is a closed set with respect to the Euclidean metric, we get that
\begin{eqnarray}
(gz, gz) \in (gA_{1} \times gA_{2}) \cup (gA_{2} \times gA_{1}),
\end{eqnarray}
thereby implying that $gz \in g(A_{1}) \cap g(A_{2})$. Therefore, we obtain immediately from the definition of $\alpha$ that $\alpha(gx_{n}, gz) \geq 1$ for all $n$. \\
Now, let $a$ be an arbitrary point in $A_{1}$. We need to show that $\alpha(ga, fa) \geq 1$. Indeed, from condition (ii), we have $fa \in g(A_{2})$. Since $ga \in g(A_{1})$, we get $(ga, fa) \in g(A_{1}) \times g(A_{2})$, which implies that $\alpha(ga, fa) \geq 1$. \\
Now, all the hypotheses of Theorem \ref{main} are satisfied. Hence, we deduce that $f$ and $g$ have a coincidence point $z \in A_{1} \cup A_{2}$, that is, $fz = gz$. If $z \in A_{1}$, from (ii), $fz \in g(A_{2})$. On the other hand, $fz = gz \in g(A_{1})$. Then, we have $gz \in g(A_{1}) \cap g(A_{2})$, which implies from the one-to-one property of $g$ that $z \in A_{1} \cap A_{2}$. Similarly, if $z \in A_{2}$, we obtain that $z \in A_{1} \cap A_{2}$. \\
Notice that if $x$ is a coincidence point of $f$ and $g$, then $x \in A_{1} \cap A_{2}$. Finally, let $x ,y \in C(g, f)$, that is, $x, y \in A_{1} \cap A_{2}$, $gx = fx$ and $gy = fy$. Now, from above observation, we have $w = x \in A_{1} \cap A_{2}$, which implies that $gw \in g(A_{1} \cap A_{2}) = g(A_{1}) \cap g(A_{2})$ due to the fact that $g$ is one-to-one. Then, we get that $\alpha(gx, gw) \geq 1$ and $\alpha(gy, gw) \geq 1$. Then our claim holds. \\
Now, all the hypotheses of Theorem \ref{main2} are satisfied. So, we deduce that $z = A_{1} \cap A_{2}$ is the unique common fixed point of $f$ and $g$. This completes the proof.
\end{proof}

The following results are immediate consequences of Corollary \ref{cyclic}.

\begin{cor}
Let $(X, d)$ be a complete metric space, $A_{1}$ and $A_{2}$ are two nonempty closed subsets of $X$ and $f, g : Y \rightarrow Y$ be two mappings, where $Y = A_{1} \cup A_{2}$. Suppose that the following conditions hold:\\
(i) $g(A_{1})$ and $g(A_{2})$ are closed;\\
(ii) $f(A_{1}) \subseteq g(A_{2})$ and $f(A_{2}) \subseteq g(A_{1})$;\\
(iii) $g$ is one-to-one;\\
(iv) there exists a function $\psi \in \Psi$ such that
\begin{eqnarray*}
d(fx, fy) \leq \psi(d(gx, gy)), \hspace*{0.1cm} \forall (x, y) \in A_{1} \times A_{2}.
\end{eqnarray*}
Then, $f$ and $g$ have a coincidence point $z \in A_{1} \cap A_{2}$. Further, if $f$, $g$ commute at their coincidence points, then $f$ and $g$ have a unique common fixed point that belongs to $A_{1} \cap A_{2}$.
\end{cor}

\begin{cor}
Let $(X, d)$ be a complete metric space, $A_{1}$ and $A_{2}$ are two nonempty closed subsets of $X$ and $f, g : Y \rightarrow Y$ be two mappings, where $Y = A_{1} \cup A_{2}$. Suppose that the following conditions hold:\\
(i) $g(A_{1})$ and $g(A_{2})$ are closed;\\
(ii) $f(A_{1}) \subseteq g(A_{2})$ and $f(A_{2}) \subseteq g(A_{1})$;\\
(iii) $g$ is one-to-one;\\
(iv) there exists a constant $\lambda \in (0, 1)$ such that
\begin{eqnarray*}
d(fx, fy) \leq \lambda \max \left\{d(gx, gy), \frac{d(gx, fx)+ d(gy, fy)}{2}, \frac{d(gx, fy)+ d(gy, fx)}{2} \right\} \hspace*{0.1cm} \forall (x, y) \in A_{1} \times A_{2}.
\end{eqnarray*}
Then, $f$ and $g$ have a coincidence point $z \in A_{1} \cap A_{2}$. Further, if $f$, $g$ commute at their coincidence points, then $f$ and $g$ have a unique common fixed point that belongs to $A_{1} \cap A_{2}$.
\end{cor}

\begin{cor}
Let $(X, d)$ be a complete metric space, $A_{1}$ and $A_{2}$ are two nonempty closed subsets of $X$ and $f, g : Y \rightarrow Y$ be two mappings, where $Y = A_{1} \cup A_{2}$. Suppose that the following conditions hold:\\
(i) $g(A_{1})$ and $g(A_{2})$ are closed;\\
(ii) $f(A_{1}) \subseteq g(A_{2})$ and $f(A_{2}) \subseteq g(A_{1})$;\\
(iii) $g$ is one-to-one;\\
(iv) there exists a constant $\lambda \in (0, 1)$ such that
\begin{eqnarray*}
d(fx, fy) \leq Ad(gx, gy) + B[d(gx, fx)+ d(gy, fy)] + C[d(gx, fy)+ d(gy, fx)], \hspace*{0.1cm} \forall (x, y) \in A_{1} \times A_{2}.
\end{eqnarray*}
Then, $f$ and $g$ have a coincidence point $z \in A_{1} \cap A_{2}$. Further, if $f$, $g$ commute at their coincidence points, then $f$ and $g$ have a unique common fixed point that belongs to $A_{1} \cap A_{2}$.
\end{cor}

\begin{cor}
Let $(X, d)$ be a complete metric space, $A_{1}$ and $A_{2}$ are two nonempty closed subsets of $X$ and $f, g : Y \rightarrow Y$ two mappings, where $Y = A_{1} \cup A_{2}$. Suppose that the following conditions hold:\\
(i) $g(A_{1})$ and $g(A_{2})$ are closed;\\
(ii) $f(A_{1}) \subseteq g(A_{2})$ and $f(A_{2}) \subseteq g(A_{1})$;\\
(iii) $g$ is one-to-one;\\
(iv) there exists a constant $\lambda \in (0, 1)$ such that
\begin{eqnarray*}
d(fx, fy) \leq \lambda(d(gx, gy)), \hspace*{0.1cm} \forall (x, y) \in A_{1} \times A_{2}.
\end{eqnarray*}
Then, $f$ and $g$ have a coincidence point $z \in A_{1} \cap A_{2}$. Further, if $f$, $g$ commute at their coincidence points, then $f$ and $g$ have a unique common fixed point that belongs to $A_{1} \cap A_{2}$.
\end{cor}

\begin{cor}
Let $(X, d)$ be a complete metric space, $A_{1}$ and $A_{2}$ are two nonempty closed subsets of $X$ and $f, g : Y \rightarrow Y$ be two mappings, where $Y = A_{1} \cup A_{2}$. Suppose that the following conditions hold:\\
(i) $g(A_{1})$ and $g(A_{2})$ are closed;\\
(ii) $f(A_{1}) \subseteq g(A_{2})$ and $f(A_{2}) \subseteq g(A_{1})$;\\
(iii) $g$ is one-to-one;\\
(iv) there exists a constant $\lambda \in (0, 1)$ such that
\begin{eqnarray*}
d(fx, fy) \leq \lambda[d(gx, fx) + d(gy, fy)], \hspace*{0.1cm} \forall (x, y) \in A_{1} \times A_{2}.
\end{eqnarray*}
Then, $f$ and $g$ have a coincidence point $z \in A_{1} \cap A_{2}$. Further, if $f$, $g$ commute at their coincidence points, then $f$ and $g$ have a unique common fixed point that belongs to $A_{1} \cap A_{2}$.
\end{cor}

\begin{cor}
Let $(X, d)$ be a complete metric space, $A_{1}$ and $A_{2}$ are two nonempty closed subsets of $X$ and $f, g : Y \rightarrow Y$ be two mappings, where $Y = A_{1} \cup A_{2}$. Suppose that the following conditions hold:\\
(i) $g(A_{1})$ and $g(A_{2})$ are closed;\\
(ii) $f(A_{1}) \subseteq g(A_{2})$ and $f(A_{2}) \subseteq g(A_{1})$;\\
(iii) $g$ is one-to-one;\\
(iv) there exists a constant $\lambda \in (0, 1)$ such that
\begin{eqnarray*}
d(fx, fy) \leq \lambda[d(gx, fy) + d(gy, fx)], \hspace*{0.1cm} \forall (x, y) \in A_{1} \times A_{2}.
\end{eqnarray*}
Then, $f$ and $g$ have a coincidence point $z \in A_{1} \cap A_{2}$. Further, if $f$, $g$ commute at their coincidence points, then $f$ and $g$ have a unique common fixed point that belongs to $A_{1} \cap A_{2}$.
\end{cor}

\textbf{Remarks}
\begin{itemize}
            \item Letting $g = I_{X}$ in Corollary 4.18 we obtain Corollary 3.19 in \cite{KarapinarandSamet}.
            \item Letting $g = I_{X}$ in Corollary 4.19 we obtain Corollary 3.20 in \cite{KarapinarandSamet}.
            \item Letting $g = I_{X}$ in Corollary 4.20 we obtain Corollary 3.21 in \cite{KarapinarandSamet}.
            \item Letting $g = I_{X}$ in Corollary 4.21 we obtain Corollary 3.22 in \cite{KarapinarandSamet}.
            \item Letting $g = I_{X}$ in Corollary 4.22 we obtain Corollary 3.23 in \cite{KarapinarandSamet}.
            \item Letting $g = I_{X}$ in Corollary 4.23 we obtain Corollary 3.24 in \cite{KarapinarandSamet}.
            \end{itemize}

\section{Acknowledgement}
The first author gratefully acknowledges the University Grants Commission, Government of India for financial support during the preparation of this manuscript.


\begin{thebibliography}{30}
\baselineskip 17pt

\bibitem{Banach}Banach, S.: Surles operations dans les ensembles abstraits et leur
application aux equations itegrales, Fundamenta Mathematicae \textbf{3}, 133-181 (1922).

\bibitem{Caccioppoli}Caccioppoli, R.: Un teorema generale sull'esistenza di elementi uniti in una trasformazione funzionale,
Rendicontilincei: Matematica E Applicazioni. \textbf{11}, 794–-799 (1930). (in Italian).

\bibitem{Kannan}Kannan, R.: Some results on fixed points, Bull. Calcutta. Math. Soc. \textbf{10}, 71-76 (1968).

\bibitem{BhaskarandLakshmikantham}Bhaskar, T. G., Lakshmikantham, V.:Fixed
Point Theory in partially ordered metric spaces and applications, Nonlinear Analysis \textbf{65}, 1379-1393 (2006).

\bibitem{Branciari}Branciari, A.:A fixed point theorem for mappings satisfying a general contractive condition of integral type,
Int. J. Math. Math. Sci. \textbf{29},  531–-536 (2002).

\bibitem{Lakshmikanthamandciric}Lakshmikantham, V., $\acute{C}$iri$\acute{c}$, L.:Coupled fixed point theorems for nonlinear contractions in
partially ordered metric spaces, Nonlinear Analysis \textbf{70}, 4341--4349 (2009).

\bibitem{Nieto5}Nieto, J. J., Lopez, R. R.:Contractive mapping theorems in partially ordered sets and applications to ordinary differential
equations, Order \textbf{22}, 223-239 (2005).

\bibitem{Saadati}Saadati, R., Vaezpour, S. M.:Monotone generalized weak contractions in partially ordered metric spaces,
Fixed Point Theory \textbf{11}, 375–-382 (2010).

\bibitem{SametVetro}Samet, B., Vetro, C., Vetro, P.: Fixed point theorem for
$\alpha$-$\psi$ contractive type mappings, Nonlinear Anal. \textbf{75},
2154-2165 (2012).

\bibitem{KarapinarandSamet}Karapinar, E., Samet, B.:Generalized
$\alpha$-$\psi$-contractive type mappings and related fixed point theorems with
applications, Abstract and Applied Analysis \textbf{2012} Article ID 793486, 17 pages
doi:10.1155/2012/793486.

\bibitem{CiricMonotone}$\acute{C}$iri$\acute{c}$, L., Cakic, N., Rajovic, M., Ume, J.S.: Monotone
generalized nonlinear contractions in partially ordered metric spaces, Fixed Point Theory Appl. 2008(2008), Article ID 131294, 11 pages.

\bibitem{AydiCone}Aydi, H., Nashine, H.K., Samet, B., Yazidi, H.: Coincidence and common
fixed point results in partially ordered cone metric spaces and applications to integral equations, Nonlinear Analysis \textbf{74}, 6814-6825 (2011).

\bibitem{Berinde}Berinde, V.:Iterative Approximation of fixed points, Editura Efemeride, Baia Mare, 2002.

\bibitem{Ciric72}$\acute{C}$iri$\acute{c}$, L.: Fixed points for generalized multi-valued mappings, Mat. Vesnik.  \textbf{24}, 265-272 (1972).

\bibitem{Hardy}Hardy, G. E., Rogers, T. D.: A generalization of a fixed point theorem of Reich, Canad. Math. Bull. \textbf{16}, 201-206 (1973).

\bibitem{Chatterjea}Chatterjea, S.K.: Fixed point theorems, C. R. Acad. Bulgare Sci. \textbf{25}, 727-730 (1972).

\bibitem{Turinici}Turinici, M.: Abstract comparison principles and multivariable
Gronwall-Bellman inequalities, J. Math. Anal. Appl. \textbf{117}, 100--127 (1986).

\bibitem{Ran}Ran, A. C. M., Reurings, M. C. B.: A fixed point theorem in
partially ordered sets and some applications to matrix equations,
Proc. Amer. Math. Soc. \textbf{132}, 1435-1443 (2004).

\bibitem{Nieto7}Nieto, J. J., Lopez, R. R.:Existence and uniqueness of fixed point in partially ordered sets and applications to ordinary
differential equations, Acta Math. Sinica, Eng. Ser. \textbf{23} 2205-2212 (2007).

\bibitem{Agarwal}Agarwal, R. P., El-Gebeily, M. A., Regan, D. O':Generalized contractions in partially ordered metric spaces, Applicable
Analysis \textbf{87}, 1-8 (2008).

\bibitem{SametMeir}Samet, B.: Coupled fixed point theorems for a generalized Meir-Keeler contraction in partially ordered metric spaces, Nonlinear Anal. TMA
(2010) doi:10.1016/j.na.2010.02.026

\bibitem{Kirkcyclic}Kirk, W. A., Srinivasan, P. S., Veeramani, P.: Fixed points
for mappings satisfying cyclical contractive conditions, Fixed Point Theory \textbf{4}, 79–-89 (2003).


\bibitem{Agarwalcyclic}Agarwal, R. P., Alghamdi, M. A., Shahzad, N.: Fixed point
theory for cyclic generalized contractions in partial metric spaces, Fixed Point Theory Appl. (2012), 2012:40.


\bibitem{Karapinarcyclic1}Karapinar E., Fixed point theory for cyclic weak $\phi$-contraction, Appl. Math. Lett. \textbf{24}, 822-825 (2011).

\bibitem{Karapinarcyclic2}Karapinar, E. and Sadaranagni, K.: Fixed point theory for cyclic $(\phi-\psi)$-contractions, Fixed point theory Appl. 2011,
2011:69.

\bibitem{Pacurar}Pacurar, M., Rus I. A.:Fixed point theory for cyclic $\varphi$-contractions, Nonlinear Anal. \textbf{72}, 1181-1187 (2010).

\bibitem{Petric}Petric, M. A.:Some results concerning cyclic contractive mappings, General Mathematics \textbf{18}, 213-226 (2010).


\bibitem{RusCyclic}Rus, I. A.:Cyclic representations and fixed points, Ann. T. Popovicin. Seminar Funct. Eq. Approx. Convexity \textbf{3}, 171-178 (2005).





\end{thebibliography}
\end{document}